\documentclass[12pt]{amsart}
\usepackage{eucal,times,amsmath,amsthm,amssymb,mathrsfs,stmaryrd,color,enumerate,accents}
\usepackage[all]{xy}
\usepackage{url}
\usepackage{fullpage}

\usepackage{float}
\usepackage{graphicx}
\usepackage{amscd}
\usepackage{amsthm}
\usepackage{amsxtra}
\usepackage{verbatim}
\usepackage{endnotes}
\usepackage{times}
\usepackage{mathtools}
\usepackage[colorlinks=true,hyperindex]{hyperref}
\theoremstyle{plain}
\newtheorem{X}{X}[section]
\newtheorem{theorem}[X]{Theorem}
\newtheorem{proposition}[X]{Proposition}
\newtheorem{lemma}[X]{Lemma}
\newtheorem{corollary}[X]{Corollary}

\theoremstyle{definition}
\newtheorem{definition}[X]{Definition}

\newtheorem{remark}[X]{Remark}

\newtheorem{question}[X]{Question}

\begin{document}

\title{Higher Euler characteristics:\\ \emph{variations on a theme of Euler}}
\date{\today}
\address{Department of Mathematics, University of Maryland, College Park, MD 20742 USA.} 
\email{atma@math.umd.edu}
\author{Niranjan Ramachandran}
\urladdr{http://www2.math.umd.edu/~atma/}
\thanks{Partly supported by the 2015-2016 ``Research and Scholarship Award''  from the
Graduate School, University of Maryland.}

\begin{abstract}
We provide a natural interpretation of the secondary Euler characteristic and introduce higher Euler characteristics. For a compact oriented manifold of odd dimension, the secondary Euler characteristic recovers the Kervaire semi-characteristic. We prove basic properties of the higher invariants and illustrate their use.  We also introduce motivic variants.
\end{abstract}
\maketitle
\begin{quote}
``{\it Being trivial is our most dreaded pitfall.    ...

``Trivial'' is relative. Anything grasped as long as two minutes ago seems trivial to a working mathematician.}''
\end{quote}
\begin{flushright} -----\href{http://www.ihes.fr/~gromov/PDF/autobiography-dec20-2010.pdf}{M. Gromov, {\it A few recollections}, 2011.}
\end{flushright}
The characteristic introduced by L.~Euler \cite{euler, euler3, euler2}, (first mentioned in a  letter to C.~Goldbach dated 14 November 1750) via his celebrated formula $$V - E + F =2,$$ is a basic ubiquitous invariant of topological spaces; two extracts from the letter:
\begin{quote} ...Folgende Proposition aber kann ich nicht recht rigorose demonstriren $$\cdots H+S = A+2.$$ 

...Es nimmt mich Wunder, dass diese allgemeinen proprietates in der Stereometrie noch von Niemand, so viel mir bekannt, sind angemerkt worden; doch viel mehr aber, dass die f\"urnehmsten davon als theor. 6 et theor. 11 so schwer zu beweisen sind, den ich kann dieselben noch nicht so beweisen, dass ich damit zufrieden bin....
\end{quote} 
When the Euler characteristic of a topological space vanishes, then it becomes necessary to introduce other invariants to study it. For instance, an odd-dimensional compact oriented manifold has zero Euler characteristic; an important invariant of such manifolds is the semi-characteristic of M. Kervaire \cite{kervaire}.

In recent years, the ``secondary" or ``derived" Euler characteristic $\chi'$ has made its appearance in many disparate fields  \cite{bismutz, farber, grayson, bunkeo, anton, lichtenbaum, bloch}; in fact, this secondary invariant dates back to 1848 when it was introduced by A.~Cayley (in a paper ``\emph{A Theory of elimination}", see \cite[Corollary 15 on p.~486 and p.~500, Appendix B]{MR2394437}).   

In this short paper, we provide a natural interpretation and generalizations of the ``secondary" Euler characteristic. Our initial aim was to understand the appearance of the ``secondary" Euler characteristic in formulas for special values of zeta functions \cite{lichtenbaum} (see \S\ref{zeta}).  We introduce invariants $\chi_j$ for $j \ge 0$ which generalize $\chi$ and $\chi'$; one has $\chi_0 = \chi$ and $\chi_1 = \chi'$; further, we prove (Corollary \ref{natural}) that $\chi_j$ is the $j^{\text{\tiny th}}$ coefficient of the Taylor expansion of the Poincar\'e polynomial $P(t)$ at $t =-1$. This interpretation seems new in the literature.

As motivation for higher Euler characteristics, consider the following questions:
\begin{itemize} 
\item  {\bf Q1} Given a compact manifold $M$ of the form $$M = N \times \underbrace{S^1 \times S^1 \cdots S^1}_{\text{r factors}}, \qquad r>0$$ which  topological invariant detects the integer $r>0$? The Euler characteristic of $M$ is always zero: $\chi(M) = \chi(N). \chi(S^1)^r =0$ independent of $r$. A related question: given $M$, how to compute the Euler characteristic of $N$?

\item {\bf Q2} For a commutative ring $A$, write $K_0(A)$ for the Grothendieck group of the (exact) category $\text{Mod}_A$ of finitely generated projective $A$-modules.  Any bounded complex $C$ of finitely generated projective $A$-modules defines a class $[C] \in K_0(A)$. As the class $[C]$ of an acyclic complex $C$ is zero, one can ask:  Are there natural non-trivial invariants  of acyclic complexes $C$? Are there enough to help distinguish an acyclic complex from a tensor product (itself acyclic) of acyclic complexes? 
\end{itemize} 
The higher Euler characteristics answer these questions; these invariants are ``special values" of the Poincar\'e polynomial; see Remark \ref{analogy-zeta}. We show (Lemma \ref{ker}) that the secondary Euler characteristic recovers the semi-characteristic of M.~Kervaire \cite{kervaire}. The topological and the K-theoretic versions of the higher Euler characteristics are in the first and third section. The last section indicates certain generalizations in the context of motivic measures and raises related questions. The second section is a gallery of secondary Euler characteristics.  

Note the analogy between taking a product with a circle $X \mapsto X \times S^1$  and taking the cone $CN$ of a self-map $N \to N$ (compare part (iii) of Theorems \ref{main1} and \ref{main2}).  J.~Rosenberg alerted us to a definition of ``higher Euler characteristics" due to R.~Geoghegan and A.~Nicas \cite{MR1341940}; the relations with this paper will be explored in future work.

\noindent {\bf Notations.} A nice topological space is, or is homotopy equivalent to, a finite CW complex.

\section{Topological setting}\label{topos}

\subsection*{Introduction.}  Recall that, for any nice topological space $M$, its Euler characteristic $$\chi(M) = \sum_i (-1)^i b_i(M)$$ is the alternating sum of the Betti numbers $b_i = b_i(M) = {\rm rank}_{\mathbb Z}~H_i(M, \mathbb Z)$.  The ``secondary" Euler characteristic of $M$ is defined as $$\chi'(M) = \sum_{i}(-1)^{i-1} i b_i = b_1 - 2b_2 +3b_3 - \cdots.$$
The topological invariant $\chi$ satisfies (and is characterized by) the following properties: it is invariant under homotopy,  $\chi(\text{point}) =1$, and, for nice spaces $U$ and $V$, \begin{align}\label{euler} \chi( U \times V) &= \chi(U). \chi(V),\nonumber \\ 
\chi(U\cup V) &=  \chi(U) + \chi(V) - \chi(U\cap V). \end{align} Clearly, $\chi'$ cannot satisfy the same properties. One has that
\begin{itemize}
\item  $\chi'$ is invariant under homotopy, 
\item $\chi'(\text{point}) =0$ 
\item but, in general, $\chi'(U\times V) \neq \chi'(U).\chi'(V)$ and 
\item $\chi'$ satisfies (\ref{euler}) only for disjoint unions. 
\end{itemize} 

As $\chi(M) =0$ (Poincar\'e duality) for any oriented compact closed manifold $M$ of odd dimension,  $\chi'(M)$ is the simplest nontrivial natural topological invariant for such manifolds.  

\begin{lemma}\label{lem} Let $M$ and $N$ be nice topological spaces.

(i)  $\chi'(M \times S^1) = \chi(M)$.

(ii) $\chi'(M \times N) = \chi(M) \chi'(N) + \chi(N). \chi'(M)$.

\end{lemma} 
 \begin{proof}  (i) This is just direct computation: Let $b_i$ be the Betti numbers of $M \times S^1$ and $c_i$ the Betti numbers of $M$. By the K\"unneth theorem, one has $b_{i+1} = c_{i+1} + c_i$. Therefore, 
\begin{align*}
\chi'(M \times S^1) & = 0. b_0 + b_1 -2 b_2 +3b_3 - \cdots\\
& = (c_1 + c_0) - 2(c_2 + c_1) +3(c_3 + c_2) \cdots\\
& = c_0 -c_1 +c_2 - \cdots \\
&= \chi(M). \qed
\end{align*}

(ii)  Direct computation. For a conceptual proof, see the proof of part (ii) of Theorem \ref{main1}.  \end{proof}

\subsection*{Kervaire's semi-characteristic.} Let $M$ be a compact oriented  manifold of odd dimension $2n+1$. Since $\chi(M) =0$, Kervaire's \cite{kervaire} semi-characteristic $$K_M = \sum_{i=0}^{i=n} (-1)^ib_i(M) \bmod 2$$ is an important topological invariant of such manifolds.  The following observation, while simple, seems new: the secondary Euler characteristic of $M$ recovers $K_M$.
\begin{lemma}\label{ker} $\chi'(M) \equiv K_M \bmod 2$. \end{lemma} 

\begin{proof} Clearly $ \chi'(M) =\sum_{i=0}^{i=n} b_{2i +1} \bmod 2$.  If $n =2k+1$, then using $(b_i = b_{4k+3-i})$ 
we have $\chi'(M) = b_1 + b_3 + \cdots b_n + b_{2k} + b_{2k-2} + \cdots b_0 \bmod 2$.  So $\chi'(M) = K_M \bmod 2$ in this case. If $n =2k$, then using $b_i = b_{4k+1 -i}$, we have $\chi'(M) 
 = b_1 + b_3 + \cdots b_{2k-1} + b_{2k} + b_{2k-2} + \cdots b_0 \bmod 2$. So $\chi'(M) = K_M \bmod 2$.  \end{proof}

\subsection*{Basic definitions and results.} 

\begin{definition} (Higher Euler characteristics) For any nice topological space $M$ and for any integer $j \ge 0$, we define the $j$'th Euler characteristic of $M$ as \begin{equation}\label{chi-j}
\chi_j(M) = \sum_{i} (-1)^{i-j} ~\binom{i}{j} ~b_i . \end{equation}
\end{definition} 
Clearly, $\chi_0(M) = \chi(M)$ and $\chi_1(M) = \chi'(M)$. If $M$ is a manifold of dimension $N$, then $\chi_j(M) =0$ for $j >N$. 
Note that $\chi_j(S^1) =0$ for $j \neq 1$ and $\chi_1(S^1) =1$. 
The higher Euler characteristics\footnote{The ``secondary" Euler characteristic is the first higher characteristic.}
 share many of the properties of $\chi$ and $\chi'$. 

\begin{theorem}\label{main1}

(i) $\chi_j$ is invariant under homotopy; for a disjoint union $U \amalg V$, one has $\chi_j(U \amalg V) = \chi_j(U) + \chi_j(V)$ and $\chi_j(\text{point}) =0$ for $j >0$.

(ii) If $\chi_r(M)$ and $\chi_r(N)$ vanish for $0 \le r <j$, then \begin{align*} \chi_k(M \times N) = 0, & \qquad {\rm for} ~0 \le k <2j\\
\chi_{2j}(M \times N) &= \chi_j(M).\chi_j(N),\\ 
\chi_{2j+1}(M \times N) &= \chi_j(M). \chi_{j+1}(N) + \chi_{j+1}(M).\chi_j(N).\end{align*} 

(iii) Let $M= N \times (\underset{j~ \text{factors}}{S^1 \times \cdots \times S^1})$. Then $\chi_j(M) = \chi(N)$ and $\chi_{k+j}(M) = \chi_k(N)$ for $k \ge 0$, $$\chi_0(M) = 0, \cdots, \chi_{j-1}(M) =0.$$    
\end{theorem}

\begin{remark}  There are at least two natural choices for the definition of the higher Euler characteristics; for $\chi_2(M)$, one could take either $\sum_i(-1)^i i^2 b_i$ or $\sum_i (-1)^i i(i-1) b_i$. More generally, an alternate definition is given by $$\chi_j(M) = \sum_i (-1)^{i-j} i^j b_i(X).$$ One has $\chi_0(X) = \chi(X)$ and $\chi_1(X) = \chi'(X)$. \qed \end{remark} 
Our proof of Theorem \ref{main1} is based on a natural interpretation of $\chi_j$'s provided by Lemma \ref{lem2} and Corollary \ref{natural}.  It is unclear if there is a simple natural proof of Theorem \ref{main1} which completely avoids this new interpretation of $\chi_j$. 

\begin{lemma}\label{lem2} For any polynomial $P(t) = \sum_i b_i t^i \in \mathbb Z[t]$, 
consider the expansion of $P(t)$ about $t=-1$, namely, define $Q(u) = \sum_i a_i u^i \in \mathbb Z[u]$ by  
$$P_M(t) = Q_M(1+t),$$\begin{equation}\label{pab}P(t) = b_0 +b_1t + b_2t^2 + \cdots = a_0 + a_1 (1+t) + a_2 (1+t)^2 + \cdots . \end{equation}
For any $j \ge 0$,  one has $$a_j = \sum_i (-1)^{i-j} \binom{i}{j} b_i.$$
\end{lemma}  

\begin{proof}  Evaluating both sides of (\ref{pab}) at $t=-1$ gives 
$$ a_0 = b_0 -b_1 + \cdots = \sum_i (-1)^i b_i= P(-1).$$
Taking the formal derivative of (\ref{pab}) with respect to $t$ gives \begin{equation}\label{pab2} b_1 + 2b_2t + 3b_3 t^2 + \cdots =  a_1 + 2a_2 (1+t) + 3 a_3(1+t)^2 + \cdots . \end{equation}
Evaluating at $t=-1$ gives $$a_1 = b_1 -2b_2 + 3b_3 - \cdots = \sum _i (-1)^{i-1} i b_i.$$
Applying $\frac{d}{dt}$ to (\ref{pab2}) gives \begin{equation}\label{pab3} 
2b_2 + 6b_3t + \cdots + n(n-1) b_n t^{n-2} + \cdots = 2a_2 + 8a_3(1+t) + \cdots + n(n-1) a_n (1+t)^{n-2} + \cdots . \end{equation} 
Plugging in $t=-1$ gives $$2a_2 = 2b_2 -6b_3 + 12b_4 + \cdots + n(n-1)b_n (-1)^{n-2} + \cdots$$ and so $$a_2 = \sum_i (-1)^{i-2} \binom{i}{2}b_i.$$Iterating these steps (apply $\frac{d}{dt}$ and evaluate at $t=-1$)  provides the required relation for any $a_j$. \end{proof}

\begin {corollary}\label{natural}  For any nice topological space $M$, write the Poincar\'e polynomial $P_M(t) = \sum_i b_i(M)t^i$ as a function of $u= 1+t$, i.e., define $Q_M(u) \in \mathbb Z[u]$ by  $P_M(t) = Q_M(1+t).$ Then, $$Q_M(u) = \sum_j ~\chi_j(M)~u^j.$$ 
\end{corollary} 
This shows that the higher Euler characteristics form a natural generalization of the Euler characteristic: $\chi_M = P_M(-1)$ and $\chi_1(M), \chi_2(M), \cdots$ {\em are the coefficients of the Taylor expansion of  $P_M(t)$ at $t=-1$}.  

\begin{proof} (of Theorem \ref{main1})  (i) the first statement is clear as the Betti numbers are homotopy invariant. For the second, use $P_{U \amalg V} (t) = P_U(t) + P_V(t)$.  

(ii) Since $P_{M \times N}(t) = P_M(t). P_N(t)$ (K\"unneth), so $Q_{M \times N}(u) = Q_M(u). Q_N(u)$.  Now apply Lemma \ref{lem2}. We are given that $Q_M(u)$ and $Q_N(u)$ are both divisible by $u^{j}$. So $Q_{M \times N}(u)$ is divisible by $u^{2j}$. As $Q_M(u) = u^j (\chi_j(M) + \chi_{j+1}(M) u + \cdots)$ and $Q_N(u) = u^j(\chi_j(N) + \chi_{j+1}(N) u + \cdots)$, we have $Q_{M \times N}(u) = u^{2j} (\chi_j(M). \chi_j(N) + (\chi_j(M).\chi_{j+1}(N) + \chi_{j+1}(M).\chi_j(N)) u + \cdots)$.  
Now apply Lemma \ref{lem2}.

(iii) By K\"unneth, one has $P_M(t) = P_N(t) (1+t)^j$ and so $Q_M(u) = Q_N(u) u^j$. Now apply Lemma \ref{lem2}. 
\end{proof}

\begin{remark}\label{analogy-zeta} The higher Euler characteristics are {\it special values} of the Poincar\'e polynomial, in the following sense.

 (a)  For any scheme $X$ of finite type over Spec~$\mathbb Z$, one introduces the analytic function $\zeta_X(s)$ (the zeta function of $X$). Conjecturally, there is arithmetic information in the special values of  $\zeta_X(s)$  at $s=n \in \mathbb Z$; if $\zeta_X(n) =0$, then one looks at  the leading term in the Taylor expansion of $\zeta_X(s)$ about $s=n$. This leading term is called a ``special value" of $\zeta_X(s)$ at $s=n$. In our context, Lemma \ref{lem2} tells us that the Euler characteristic is the value of the Poincar\'e polynomial $P(t)$ at $t=-1$ and the  ``special values" of $P(t)$ at $t=-1$ are the higher Euler characteristics.

(b) Part (iii) of Theorem \ref{main1} provides a partial answer to {\bf Q1} posed above. Namely, each factor of $S^1$ in $M$ causes the vanishing of a higher Euler characteristic of $M$. Thus, the number of factors $r$ of $S^1$ in $M$ satisfies the inequality $$ r \le {\rm ord}_{~t=-1}~ P_M(t) = {\rm ord}_{~u=0}~Q_M(u)$$ with equality if and only if $\chi(N) \neq 0$. 
So  $\chi_0(M) =0, \cdots, \chi_{r-1}(M) =0$. 

(c) The higher Euler characteristics do not satisfy (\ref{euler}) in general;  this follows from the fact that, in general, $P_{U\cup V} \neq P_U(t) + P_V(t) - P_{U \cap V}(t)$. 

(d) There is a straightforward generalization of higher Euler characteristics of local systems (or sheaves) on nice topological spaces (or algebraic varieties). Namely, given a local system $A$ of say $\mathbb Q$-vector spaces on a nice space $X$,  the cohomolology groups $H^i(X, A)$ are finite dimensional $\mathbb Q$-vector spaces. The higher Euler characteristics $\chi_j(X, A)$ are the coefficients of the Taylor expansion about $t=-1$ of the Poincar\'e polynomial $P_X(A, t) = \sum_i {\rm dim}_{\mathbb Q}~H^i(X, A)~t^i$. When $A =\mathbb Q$ is  the trivial local system, one recovers the usual higher Euler characteristics: $\chi_j(X, A) = \chi_j(X)$.   Similarly, if $\mathcal F$ is a coherent sheaf on a proper variety $X$ over a field $K$, then the cohomology groups $H^i(X, \mathcal F)$ are finite dimensional $K$-vector spaces. The associated Poincar\'e polynomial leads to the higher Euler characteristics of $\mathcal F$. In the same vein, given a $\mathbb Q_{\ell}$-constructible sheaf $\mathcal F$ on any variety $X$ over a field $K$ of characteristic different from $\ell$, the cohomology groups (with compact support) $H^i_c(X, \mathcal F)$ are finite dimensional $\mathbb  Q_{\ell}$-vector spaces; the coefficients of the Taylor expansion about $t =-1$ of the Poincar\'e polynomial $\sum_i {\rm dim}_{\mathbb Q_{\ell}}~H^i_c(X, \mathcal F) t^i$ are the higher Euler characteristics of the sheaf $\mathcal F$ over $X$. \qed
\end{remark} 

\begin{remark}  Let $X \to B$ be a fibration with fiber $F$.  Then the well known identity $\chi(X) = \chi(F)\chi(B)$ does not generalize to higher Euler characteristics.  Lemma \ref{lem} does not generalize (from products) to fibrations. For instance, consider the Hopf fibration $S^3 \to S^2$ with fibers $S^1$. Lemma \ref{lem} (part (ii)) fails in this case: $\chi'(S^3) = 3 \neq 2 \times 1 + (-2) \times 0 = \chi(S^2) \chi'(S^1) + \chi'(S^2) \chi(S^1)$. 
\end{remark}

\begin{proposition}\label{mcd} Let $M$ be a compact oriented manifold of dimension $N$. For any integer $r \ge 1$, write $Sym^r(M)$ for the $r^{\text 'th}$ symmetric product of $M$. The higher Euler characteristics $\chi_j(Sym^r(M))$ of $Sym^r(M)$ are determined by the Betti numbers $b_i(M)$ of $M$. \end{proposition} 
\begin{proof}For any nice space $X$, write  $P(X) = \sum_i (-)^i b_i(X) z^i \in \mathbb Z[z]$ for the (graded) Poincar\'e polynomial of $X$.  Recall the classical formula of  I.G.~Macdonald's \cite{macdonald, wittzeta} which show that $P(Sym^n(M))$ is determined by that of $P(M)$: 
\begin{align*} \sum_{r=0}^{\infty} P(Sym^r(M)) t^r & = \frac{ (1-zt)^{b_1(M)} (1-z^3t)^{b_3(M)} \cdots}{(1-t)^{b_0(M)} (1-z^2t)^{b_2(M)} \cdots}\\
& = \prod_{j=1}^{j=N} (1-z^jt)^{(-1)^{j+1} b_j(M)}.
\end{align*}
Since $P(M)$ determines $P(Sym^r(M))$ which in turn determines $\chi_j(Sym^r(M))$, the Betti numbers of $M$ determine the integers $\chi_j(Sym^r(M))$ for all $r, j \ge 0$.
\end{proof}
\begin{remark} Macdonald \cite{macdonald, wittzeta} also proved that $\chi(M)$ determines $\chi(Sym^r(M))$ for all $r \ge 0$: $$\sum_{r=0}^{\infty} \chi(Sym^r(M)) t^r = \frac{1}{(1-t)^{\chi(M)}};$$ it is unclear if this generalizes to $\chi_j$ for $j >0$. 

 {\em Does the integer $\chi_j(M)$ determine  the integers $\chi_j(Sym^r (M))$ for all $r \ge1$?} 
  \qed\end{remark} 
\section{Examples of secondary Euler characteristics.}\label{zeta}

For any bounded complex $C^{\bullet}$ of finitely generated abelian groups 
\[ \cdots \to 0 \to C_0 \xrightarrow{d} C_1 \xrightarrow{d} \cdots C_n \to 0 \to \cdots,\]
one defines $\chi(C^{\bullet}) = \sum_{i=0}^{i=n}(-1)^i {\rm rank} ~C_i$; it is elementary that $\chi(C^{\bullet}) = \sum_{i=0}^{i=n}(-1)^i {\rm rank} ~H_i(C^{\bullet})$. We write $\chi'(C^{\bullet}) = \sum_{i=0}^{i=n}(-1)^{i-1} i ({\rm rank}~ C_i)$; this is of interest when $\chi(C^{\bullet}) =0$. 
 
Similarly, given any abelian category $\mathcal A$ and  any bounded complex  $C^{\bullet}$ of objects in $\mathcal A$, one defines $$\chi(C^{\bullet}) = \sum_{i}(-1)^i ~[C_i]~\text{and}~\chi'(C^{\bullet})  = \sum_{i}(-1)^{i-1} i~ [C_i], $$ which are elements of $K_0(\mathcal A)$. Here $[X]$ denotes the class in $K_0(\mathcal A)$ for any object $X$ of $\mathcal A$.

Some of the well known occurences of secondary Euler characteristics include 

\begin{itemize}

\item (Ray-Singer) \cite{raysinger}  Let $M$ be a compact oriented manifold without boundary of dimension $N$. The 
Franz-Reidemeister-Milner torsion (or simply R-torsion) $\tau(M, \rho)\in \mathbb R$ is defined for any acyclic orthogonal representation 
$\rho$ of the fundamental group $\pi_1(M)$,  Let $K$ be a smooth triangulation of $M$ and $\Delta_j$ be the combinatorial
Laplacians associated with $K$ and $\rho$. Then \cite[Proposition 1.7]{raysinger} $${\rm log}~\tau(M, \rho) = \frac{1}{2} \sum_{i =0}^{i=N} (-1)^{i+1} i~{\rm log~ det}~(-\Delta_i).$$ Ray-Singer conjectured (and J.~Cheeger-W.~M\"uller proved) that this is equal to analytic torsion (which they defined in terms of a Riemannian structure on $M$).

(It is reasonable to introduce ``the torsion Poincar\'e polynomial'' 
\begin{equation}\label{torsionP}
R(M, \rho)(t) = \sum_i {\rm log~det}~(-\Delta_i) t^i \in \mathbb R[t];
\end{equation} as its Taylor expansion $R(M, \rho) = \sum_j c_j(M, \rho) (t-1)^j$ at $t=-1$ contains ${\rm log}~\tau(M, \rho)$ as $c_1(M, \rho)$, the other coefficients $c_j$ can be considered as (logarithms of) {\it higher analytic torsion}  \cite{anton} of $M$ and $\rho$.)  
\item (Lichtenbaum) \cite{lichtenbaum} For any smooth projective variety $X$ over a finite field $\mathbb F_q$, the Weil-\'etale cohomology groups $H^i_W(X , \mathbb Z)$ give a bounded complex  $C^{\bullet}$ (of finitely generated abelian groups)  $$C^{\bullet}:\qquad \cdots H^i_W(X, \mathbb Z) \xrightarrow{\cup \theta}H^{i+1}_W(X, \mathbb Z) \cdots;$$one has $\chi(C^{\bullet}) =0$ and $\chi'(C^{\bullet}) $ is the order of vanishing of the zeta function $Z(X,t)$ at $t=1$. 


\item (Grayson) \cite[\S 3, p.~103]{grayson} Let $R$ be a commutative ring and $N$ a finitely generated projective $R$-module. If $S^kCN$ is the $k$'th symmetric product of the mapping cone $CN$ of the identity map on $N,$ then Grayson's formula for the $k$'th Adams operation $\psi^k$ reads $$\psi^k[N] = \chi'(S^k CN).$$
(This raises the question: {\em Is there a natural interpretation of $\chi_j(S^k CN)$ for $j >1$?})

\item (Fried) \cite[Theorem 3]{fried} Let $X$ be a closed oriented hyperbolic manifold of dimension $2n+1>2$ and let $\rho$ be an orthogonal representation of $\pi_1(X)$. Write $V_{\rho}$ for the corresponding local system on $X$. The order of vanishing of the Ruelle zeta function $R_{\rho}(s)$ at $s=0$ is given by $$2\sum_{i=0}^{i=n}~(-1)^i~(n+1-i)~{\rm dim}~H^i(X, V_{\rho}).$$   
\item (Bunke-Olbrich) \cite{bunkeo} Given a locally symmetric space of rank one $Y = \Gamma\backslash{G}/K$ and a homogeneous 
vector bundle $V$ (this depends on a pair ${\sigma, \lambda}$ and the associated distribution globalization $V_{-\infty}$ (a complex representation of $\Gamma$), the order of vanishing of the Selberg zeta function $Z_S(s, \sigma)$ at $s =\lambda$ is given by $\chi'(\Gamma, V_{\infty})$ (Patterson's conjecture) 
$$\sum_i (-1)^{i+1} i~{\rm dim}~H^i(\Gamma, V_{\infty}).$$
\end{itemize}

\section{K-theoretic variants}\label{kth} 

K-theory provides another general context to develop higher Euler characteristics.  

As in \cite{grayson},  let $\mathcal P$ be an exact category with a suitable notion of tensor product (bi-exact), symmetric power and exterior power. Examples include the category $\mathcal P(X)$ of vector bundles over a scheme $X$, the category $\mathcal P (R)$ of finitely generated projective modules over a commutative ring $R$ and for a fixed group $\Gamma$, the category $\mathcal P(\Gamma, R)$ of representations of $\Gamma$ on finitely generated projective $R$-modules. For any object $N$ in $\mathcal P$, let us write $[N]$ for the class of $N$ in the Grothendieck ring $K_0(\mathcal P)$. For any bounded complex $M$ over $\mathcal P$, we write $\chi(M) = \sum_i (-1)^i [M_i] \in K_0(\mathcal P).$

\begin{definition} (i) The higher Euler classes $\chi_j(M)$ of $M$ are defined by $$\chi_j(M) = \sum_i (-1)^{i-j} \binom{i}{j} [M_i] \in K_0(\mathcal P), \qquad j\ge 0.$$

(ii) The Poincar\'e function $P_M(t)$ is defined as $$P_M(t) = \sum_i [M_i] t^i \in K_0(\mathcal P)[t, t^{-1}].$$\end{definition} 
Clearly, $\chi_0(M) = \chi(M)$ and $\chi_1(M) = \chi'(M)$. If $M$ is concentrated in non-negative degrees, then $P_M(t)$ is the Poincar\'e polynomial of $M$. If $M[n]$ is the shifted complex (so that $M[n]_i = M_{i +n}$), then $t^n P_{M[n]}(t) = P_M(t)$.  Defining $Q_M(u) \in K_0(\mathcal P)[u, u^{-1}]$ by $Q_M(1+t) = P_M(t)$, we have $(u-1)^n Q_{M[n]}(u) = Q_M(u)$.    

\begin{theorem}\label{main2} Let $M$ and $N$ be bounded complexes in $\mathcal P$ concentrated in non-negative degrees. 

(i) If $\chi_r(M)$ and $\chi_r(N)$ vanish for $0 \le r <j$, then \begin{align*} \chi_k(M \otimes N) = 0,  \quad {\rm for} ~0 \le k <2j &\\
\chi_{2j}(M \otimes N) = \chi_j(M).\chi_j(N),&\\ 
\chi_{2j+1}(M \otimes N) = \chi_j(M). \chi_{j+1}(N) + & \chi_{j+1}(M).\chi_j(N).\end{align*} 

(ii) If $M =CN$ is the cone of a self-map $N \to N$, then $\chi_{j+1}(M) = \chi_j(N)$ 

(iii) Let $M=  C^j(N) = C(\cdots C(N) \cdots)$ be an $j$-fold iterated cone on $N$. Then $\chi_j(M) = \chi(N)$ and $\chi_{k+j}(M) = \chi_k(N)$ for $k \ge 0$, $$\chi_0(M) = 0, \cdots, \chi_{j-1}(M) =0.$$     
\end{theorem}
It is unclear if there is a direct proof of the above theorem which does not use the interpretation of $\chi_j(M)$ as the coefficients of the Taylor expansion of $P_M(t)$ at $t=-1$. 

\begin{proof} The arguments are the same as in Theorem \ref{main1}. Defining $Q_M(u), Q_N(u) \in K_0(\mathcal P)[u]$ by $Q_M(1+t) = P_M(t)$ and $Q_N(1+t) = P_N(t)$, it follows  that $Q_M(u) = \sum_j \chi_j(M) u^j$. Now argue as in the proof of Theorem \ref{main1}. This proves (i).

(ii) As $M$ is the total complex associated with $CN$, we have  $M_0 = N_0$ and, for $i >0$,  that $M_i = N_i \oplus N_{i-1}$.   Thus $P_M(t) = P_N(t) (1+t)$ which gives $Q_M(u) =u Q_N(u)$. This proves (ii). Part (iii) follows from (ii) by induction or one can observe that $Q_M(u) = u^j Q_N(u)$.  \end{proof}
The following corollary is implicit in \cite{grayson}.
\begin{corollary} Let $M$ be a complex of $\mathcal P$ concentrated in non-negative degrees. If $\chi_0(M) \neq 0$, then $M \neq CN$. If $\chi_1(M) \neq 0$, then $M$ is not the tensor product of two acyclic complexes.
\end{corollary} 
\begin{remark} Theorem \ref{main2} is compatible with the intuition expressed in \cite[p. 104]{grayson}: 

\begin{quote} ``... we regard acyclic complexes as being infinitesimal in size when compared to arbitrary complexes.  ....that we regard doubly acyclic complexes as being doubly infinitesimal in size when compared to arbitrary complexes. It also suggests that we regard the Adams operation $\psi^k$ as being the differential of the functor $N \mapsto S^kN$ from the category of finitely generated projective modules to itself; ...''
\end{quote} 

Namely, for an acyclic complex $C$, one has $\chi_0(C) =0$ but not always $\chi_1(C) = 0$; for the tensor product $C \times D$ of acyclic complexes, one has $\chi_0(C\otimes D) = 0 = \chi_1(C\otimes D)$ but not always $\chi_2(C \otimes D) =0$. So an acyclic complex $C$ is like an infinitesimal $\epsilon$ and the tensor product $C\otimes D$ of acyclic complexes is like $\epsilon^2$, an infinitesimal of second order. The above text also suggests that for any functor $F: \mathcal P \to \mathcal P$, we regard $\chi_1(F(CN))$ as the differential of $F$  and that $\chi_j(F(-))$ as a higher differential of $F$ (when evaluated on acyclic complexes or their tensor products).  Thus, the vanishing of $\chi_0(M), \chi_1(M), \cdots \chi_n(M)$ means $M$ is like an infinitesimal $\epsilon ^n$ of order $n$. 

One possible answer to {\bf Q2} is as follows: the higher Euler characteristics $\chi_j$ provide non-trivial invariants of acyclic complexes. One has a non-trivial filtration $\tau_{\bullet}$ on the acyclic complexes in $\mathcal P$  defined for $n \ge 0$ by $\tau_n =$ the set of complexes $M$ with $\chi_0(M) =0, \chi_1(M)=0, \cdots \chi_n(M) = 0$  (order of $\epsilon^n$ or smaller).  
 \qed \end{remark} 
 
 \subsection*{Homological Poincar\'e polynomials.}  Let  $\mathcal D = D^b(\mathcal P)$  be the bounded derived category of $\mathcal P$ (now assumed to be abelian). The definition of higher Euler characteristics for $\mathcal P$ does not extend directly to the category $\mathcal D$; though Poincar\'e polynomials respect short exact sequence of complexes, they do not respect quasi-isomorphisms: 
  
  \begin{itemize} 
\item   For any short exact sequence of complexes $$0\to A \to B \to C\to 0$$ in $\mathcal P$, one has $P_B(t) = P_A(t) + P_C(t)$ and so $$\chi_j(B) = \chi_j(A) + \chi_j(C).$$
\item  Suppose that the complexes $A$ and $B$ are quasi-isomorphic.  One cannot conclude that $P_A(t) = P_B(t)$. For instance, the complex $A = N \xrightarrow{{\rm id}_N} N$ is quasi-isomorphic to the trivial complex $B$; then $P_A(t) =[N] + [N]t$ could be non-zero whereas $P_B(t)$ is always zero.  
\end{itemize} 
 
One may instead consider the {\em homological} Poincar\'e function of a bounded complex $M$ defined by
 $$P^h_M(t) = \sum_i [H^i(M)]t^i \in K_0(\mathcal P)[t, t^{-1}].$$ 
 This respects quasi-isomorphisms but not short exact sequences: 
 \begin{itemize} 
 \item  If $A$ and $B$ are quasi-isomorphic, then $P^h_A(t) = P^h_B(t)$. 
 \item If $0 \to A \to B \to C \to 0$ is a short exact sequence of complexes, then the identity $P^h_B(t) = P^h_A(t) + P^h_C(t)$ may fail to hold in general. 
 \end{itemize} 
 One can define the homological higher Euler characteristics $\chi_j^h(M)$ of $M$ as the coefficients of the Taylor expansion of $P^h_M(t)$ about $t=-1$. If the homology $h^{\bullet}(M)$ of $M$ is concentrated in non-negative degrees, one has $$\chi^h_j(M) = \sum_i (-1)^{i-j} \binom{i}{j} [h^i(M)] \in K_0(\mathcal P), \qquad j\ge 0.$$
 Theorems \ref{main1} and \ref{main2} remain valid with $\chi_j$ replaced with $\chi^h_j$. 
 
If $0 \to A \to B \to C \to 0$ is a short exact sequence of {\em acyclic} complexes,  one has $$\chi^h_1(B) = \chi_1^h(A) + \chi^h_1(C).$$ (This identity may not hold, if the acyclicity assumption is dropped.)  
  
While $\chi_0(M) = \chi^h_0(M)$ (Euler's identity), the identity $\chi_j(M) = \chi^h_j(M)$ for $j >1$ does not hold in general. This is because $P_M(t) \neq P_M^h(t)$ in general; for instance, if $A = N \xrightarrow{{\rm id}_N} N$, then $P^h_A(t) =0$ but $P_A(t)$ could be non-zero. 

Thus, it is unclear if there is a good definition of higher Euler characteristics on $\mathcal D$.
 
 \begin{remark}\label{grading}  Suppose that the category $\mathcal P$ has a $\mathbb Z$-grading; an important example is the conjectural category of motives over a given field (the theory of weights give the $\mathbb Z$-grading).
 
    For any object $M = \oplus_i M_i$, we write $M_i$ for its component of weight $i \in \mathbb Z$. The Poincar\'e function $P_M(t) \in K_0\mathcal P[t,t^{-1}]$ of $M$ is defined 
 as $P_M(t) = \sum_i [M_i] t^i$. We can define the higher Euler characteristics $\chi_j(M)$ as the coefficients of the Taylor expansion of $P_M(t)$ about $t=-1$, i.e., they are defined by the identity  $P_M(t) = \sum_j \chi_j(M) (1+t)^j$.
 \qed
 \end{remark}

\section{Final remarks}
 \subsection*{Motivic variants}  Higher Euler characteristics can be defined in a motivic context.  
 
 (i) \footnote{Motivic conjectures predict analogous results over arbitrary fields.} For any subfield $k\hookrightarrow \mathbb C$, one has the category $ \mathcal M^{AH}_k$ of absolute Hodge motives \cite[p.~5]{rollin} which has a natural $\mathbb Z$-grading coming from weights.  Any smooth proper variety $X$ over $k$ defines an object $h(X)$ of $\mathcal M^{AH}_k$; using the weight decomposition of $h(X)$, one gets, as in Remark \ref{grading}, invariants $P_{h(X)}(t)$ and $\chi_j(h(X))$; these are the motivic Poincar\'e polynomial of $X$ and the motivic higher Euler characteristics $\chi_j^{mot}(X)\in K_0 \mathcal M^{AH}_k$ of $X$.  The Betti realization gives a homomorphism 
  $r: K_0 \mathcal M^{AH}_k[t, t^{-1}] \to \mathbb Z[t, t^{-1}]$ of graded rings;  the element $r(h(X))$ is the usual Poincar\'e polynomial of the topological space $X(\mathbb C)$. Thus the motivic higher Euler characteristics of $X$ refine those of the topological space $X(\mathbb C)$.

 (ii) {\em Motivic measures} \cite{exp, wittzeta}:     Consider the category $\mathrm{Var}_F$ of varieties (integral separated schemes of finite type) over a field $F$. The Grothendieck ring $K_0 \mathrm{Var}_F$ of varieties over $F$ is defined as the quotient of the free abelian group on the set
 of isomorphism classes $[X]$ of varieties by the relations $[X] = [Y] + [X\backslash Y]$ where $Y$ is a closed subvariety of $X$. The multiplication is induced by the product of varieties. When $F$ is of positive characteristic, one needs also to impose the relation $[X]=[Y]$ for every surjective radicial morphism $X \to Y$. A {\em motivic measure} $\mu$ is a ring homomorphism $$\mu: K_0\mathrm{Var}_F \to R$$ to a ring $R$. The Euler characteristic $\chi_c$ with compact support is the prototypical motivic measure: $\chi_c: K_0\mathrm{Var}_{\mathbb C} \to \mathbb Z$ is a ring homomorphism. Another motivic measure is the Poincar\'e characteristic $\mu_P:  K_0\mathrm{Var}_{\mathbb C} \to \mathbb Z[t]$, determined by the following property: for any smooth proper variety $X$, one has $\mu_P(X) \in \mathbb Z[t]$ is the Poincar\'e polynomial of the topological space $X(\mathbb C)$. 
 
 As motivic measures are refined Euler characteristics, it is natural that certain motivic measures lead to refined higher Euler characteristics. 
 
 The usual higher Euler characteristics arise as the coefficients of the Taylor expansion about $t =-1$ of  the Poincar\'e characteristic $\mu_P: K_0\mathrm{Var}_{\mathbb C} \to \mathbb Z[t]$. This can be generalized as follows. Given a motivic measure  $\mu: K_0\mathrm{Var}_F \to A[t]$ with values in the polynomial ring over a ring $A$, we can define the higher motivic measures $\mu_j(X)$ as the coefficient of $t^j$ in the element $\mu(X) \in A[t]$. Namely, the following identity holds in $A[t]$: $$\mu(X) = \sum_j \mu_j(X) t^j.$$ While $X \mapsto \mu_0(X)$ gives the motivic measure
 $$K_0\mathrm{Var}_F \xrightarrow{\mu} A[t] \underset{t \mapsto 0}{\to} A,$$ the higher motivic measures $X \mapsto \mu_j(X)$ are just additive maps $\mu_j: K_0\mathrm{Var}_F \to A$. 
 
 More generally, one can look at the coefficients $\chi_j(X)$ of the expansion of  $\mu(X)$ at $t = a$ for any element $a$ of $A$, i.e., $$\mu(X) = \sum_j \chi_j(X) (t-a)^j \in A[t].$$ The constant term $\chi_0$ would be a motivic measure whereas the other coefficients (higher motivic characteristics) would give additive maps $\chi_j: K_0\mathrm{Var}_F \to A$. Then $\mu$ and $\chi_0$ generalize  the Poincar\'e characteristic $\mu_P$ and the Euler characteristic (with compact support) $\chi_c$ (obtained with $t =-1$).  
    
 Let us indicate another important example. The assignment $$X \mapsto \mathrm{H}_X(u,v):= \sum_{p,q\geq 0} h^{p,q}(X)u^pv^q$$ for smooth projective $X$ gives rise to the Hodge characteristic measure $$\mu_{\mathrm{H}}:K_0\mathrm{Var}_{\mathbb C} \to \mathbb Z[u,v].$$ As $\mathbb Z[u,v] = \mathbb Z[u][v]$, we take $A = \mathbb Z[u]$ and $t = v$. Let $a = u$.
 The higher motivic measures $\chi_j^{\mathrm{H}}(X)\in A$ defined by the identity  $$\mu_{\mathrm{H}}(X) = \sum_j\chi^{\mathrm{H}}_j(X) (v-u)^j$$seem to be new in the literature.   The motivic measure  $\chi_0^{\mathrm{H}}$ is the Poincar\'e characteristic: for any smooth proper variety $X$, one has $\chi_0^{\mathrm{H}}(X) = \mu_P(X)$.
 This follows from the observation that $\mathrm{H}_X(u,u) = \mu_P(X)$ (consequence of Hodge theory). 
 
 (iii) For any variety $X$ over $F$, its class $[X]$ in  $K_0\mathrm{Var}_F$ is the universal Euler characteristic (with compact support) of $X$. This motivic measure corresponds to the identity map on $K_0\mathrm{Var}_F$. Since the ring $K_0\mathrm{Var}_F$ has neither a natural grading nor a natural isomorphism with a polynomial ring, the above discussion does not  provide a definition of the universal Poincar\'e polynomial or the related universal higher Euler characteristics (as elements of $K_0\mathrm{Var}_F$). 
 
 The most natural candidate for a "universal" higher Euler characteristic with compact support is provided by the theory of Chow motives \cite{exp}, as follows: Consider the Grothendieck ring   $K_0(\mathrm{Chow}(F))$ of the rigid symmetric monoidal category ${\mathrm{Chow}}(F)$ of Chow motives over $F$ (with $\mathbb Q$-coefficients). Any smooth proper variety $X$ over $F$ defines an object $h(X)\in \mathrm{Chow}(F)$. The existence of a Chow-K\"unneth decomposition $$h(X) = \oplus_i h^i(X) \in\mathrm{Chow}(F), $$ permits the  definition of the motivic Chow-Poincar\'e polynomial $\mathcal P(X)$: $$\mathcal P(X) = \sum_i [h^i(X)] t^i \in K_0(\mathrm{Chow}(F)).$$
 The higher Chow-Euler characteristics of $X$ are the coefficients $$\chi_j^{Chow}(X) = \sum_i (-)^{i-j} ~\binom{i}{j}~[h^i(X)] \in K_0(\mathrm{Chow}(F))$$ of the expansion $$\mathcal P(X) = \sum_j \chi_j^{Chow}(X) (t+1)^j$$ of $\mathcal P(X)$ about $t=-1$. When $F$ is a subfield of $\mathbb C$, these refine the invariants defined above via absolute Hodge motives.

  \begin{remark} Suppose $\mathcal P$ is a $\mathbb Z$-graded neutral $\mathbb Q$-linear Tannakian category. For any object $M$ of $\mathcal P$, the higher 
  Euler characteristics $\chi_j({\rm Sym}^nM)\in K_0\mathcal P$ are determined by the Poincar\'e polynomial $P_M(t)$: this follows from the motivic Macdonald formula proved by  S. del Ba\~no \cite{rollin} \cite[\S 2.6]{bourqui}. From this motivic Macdonald formula, one deduces a motivic generalization of Proposition \ref{mcd}. \end{remark}

\subsection*{Finite categories}  C.~Berger and T.~Leinster \cite{berger, leinster} have provided and studied various definitions of the Euler characteristic of a finite 
category $\mathcal C$. The series Euler characteristic \cite[2.3]{berger} $\chi(\mathcal C)$ is defined to be value at $t=-1$ of a formal power series $f_{\mathcal C}(t) \in \mathbb Q(t)$. Define 
$g(u) \in \mathbb Q(u)$ by $g_{\mathcal C}(t+1) = f_{\mathcal C}(t)$; so $g(0) = \chi(\mathcal C)$. Then the higher Euler characteristics $\chi_j(C)$ of $\mathcal C$ are the coefficients of the Taylor expansion of $g(u)$ about $u=0$ (corresponding to $f_{\mathcal C}(t)$ about $t=-1$):
$$g(u) =\sum_j \chi_j(\mathcal C)u^j.$$Since $g$ could have a pole at $u=0$, this even gives a definition of lower Euler characteristics!

We end this paper with the 
\begin{question}  
\begin{enumerate} 
\item Given a ring homomorphism $f: A\to B$ between two commutative rings, consider the ideal $J$ of $K_0(A)$ defined as $$J = {\rm Ker}(f_*: K_0(A) \to K_0(B)).$$ Given a bounded complex $X$ of finitely generated projective $A$-modules whose class lies in $J$, how to determine the integer $r$ such that the class of $X$ is in $J^r -J^{r+1}$?

\item Is there an analogue of Theorem \ref{main1} for higher analytic torsion \cite{anton}? Is the analytic Poincare polynomial \eqref{torsionP}  of a product $M \times N$ determined by that of $M$ and $N$?

\item Is there an analogue of our results in the context of Kapranov's $N$-complexes \cite{kapranov}?
\end{enumerate}
 \end{question}  

\subsection*{Acknowledgements.}  {I heartily thank C.~Deninger, M.~Flach, J.~Huang, S.~Lichtenbaum, J.~Rosenberg, and L.~Washington for useful conversations and inspiration. When I communicated my ideas on the higher Euler characteristics to Deninger, he  alerted me to a beautiful paper \cite{anton} of A. Deitmar. Some of the ideas of \S 1 can also be found in \cite{anton}.  While our motivations are similar, the approach here is closer to \cite[p~104]{grayson}.  I am very grateful to the referee  for her/his very careful reading; her/his detailed comments and criticism led to this improved version of the paper. }

\begin{flushright}
Island where all becomes clear.

Solid ground beneath your feet.

The only roads are those that offer access.

Bushes bend beneath the weight of proofs.

The Tree of Valid Supposition grows here

with branches disentangled since time immemorial.

The Tree of Understanding, dazzlingly straight and simple,

sprouts by the spring called Now I Get It.

The thicker the woods, the vaster the vista:

the Valley of Obviously.

If any doubts arise, the wind dispels them instantly.

Echoes stir unsummoned

and eagerly explain all the secrets of the worlds.

On the right a cave where Meaning lies.

On the left the Lake of Deep Conviction.

Truth breaks from the bottom and bobs to the surface.

Unshakable Confidence towers over the valley.

Its peak offers an excellent view of the Essence of Things.

For all its charms, the island is uninhabited,

and the faint footprints scattered on its beaches

turn without exception to the sea.

As if all you can do here is leave

and plunge, never to return, into the depths.

Into unfathomable life.

- W.  Szymborska, {\it Utopia} (A large number, 1976)
\end{flushright}
\def\cprime{$'$} \def\cprime{$'$}

\end{document}